\documentclass{amsart}

\usepackage{mypreamble}

\newtheorem{thm}[subsubsection]{Theorem}
\newtheorem{lem}[subsubsection]{Lemma}

\newtheorem{prop}[subsubsection]{Proposition}
\newtheorem{cor}[subsubsection]{Corollary}

\theoremstyle{remark}
\newtheorem{rem}[subsubsection]{Remark}

\theoremstyle{definition}

\nop\MC{MC}
\nc\pos{\mathrm{pos}}
\nc\enh{\mathrm{enh}}
\nc\adj{\mathrm{adj}}
\nc\sub{\mathrm{sub}}
\nc\oo[1]{\overset\circ{#1}}
\nc\VinBun{\mathrm{VinBun}}

\numberwithin{equation}{section}

\title{On the reductive monoid associated to a parabolic subgroup}
\author{Jonathan Wang}

\begin{document} 

\begin{abstract}
Let $G$ be a connected reductive group over a perfect field $k$. We study 
a certain normal reductive monoid $\wbar M$ associated to a parabolic
$k$-subgroup $P$ of $G$. The group of units of $\wbar M$ is the Levi factor $M$ of $P$. 
We show that $\wbar M$ is a retract of the affine closure of the quasi-affine variety 
$G/U(P)$. Fixing a parabolic $P^-$ opposite to $P$, we prove that
the affine closure of $G/U(P)$ is a retract of the affine closure of the
boundary degeneration $(G \xt G)/(P \xt_M P^-)$. 
Using idempotents, we relate $\wbar M$ to the Vinberg semigroup of $G$. 
The monoid $\wbar M$ is used implicitly in the study of stratifications of Drinfeld's 
compactifications of the moduli stacks $\Bun_P$ and $\Bun_G$.
\end{abstract}

\maketitle

\section{Introduction}

\subsection{Motivation}

\subsubsection{}
Let $G$ be a connected reductive group over a perfect field $k$.
Let $U(P)$ denote the unipotent radical of a parabolic subgroup $P$ of $G$.
Grosshans proved in \cite{Grosshans} that the homogeneous space $G/U(P)$
is a quasi-affine variety and the algebra of regular functions $k[G/U(P)]$ 
is finitely generated. 

In \cite{AT}, Arzhantsev and Timashev consider affine embeddings of $G/U(P)$ 
and give a detailed description of the \emph{canonical embedding} $G/U(P) \into \spec k[G/U(P)]$ 
under the assumption that the characteristic of $k$ is $0$.
They establish a bijection between these affine embeddings and certain normal algebraic
monoids with group of units equal to the Levi factor $M = P/U(P)$.
In particular, the canonical embedding corresponds to the monoid $\wbar M$ defined as 
the closure of $M$ in $\spec k[G/U(P)]$.
This construction, which we first learned from \cite{Baranovsky}, defines an affine algebraic 
monoid $\wbar M$ in any characteristic. It is not \emph{a priori} clear, however,
whether the monoid $\wbar M$ is normal in positive characteristic. 

One of the goals of this paper is to show that $\wbar M$ is a normal algebraic monoid
with group of units $M$ in any characteristic, and to describe the combinatorial 
data it corresponds to under
the classification of normal reductive monoids in \cite[Theorem 5.4]{Renner}.

\subsubsection{}
Let $\wbar{G/U(P)}$ denote the spectrum of $k[G/U(P)]$. Then $\wbar{G/U(P)}$
is an affine variety of finite type, and it plays a prominent role in the definition
of Drinfeld's compactification $\wtilde \Bun_P$ of the moduli stack 
of $P$-bundles over a smooth complete curve.
Drinfeld's compactification is used to define the geometric Eisenstein series functors
in \cite{BG}. As Baranovsky observes in \cite[\S 6]{Baranovsky}, 
the monoid $\wbar M$ is used implicitly when studying the stratification
of $\wtilde \Bun_P$. More specifically, the closed subscheme $\Gr^+_M \subset \Gr_M$ of the 
affine Grassmannian (cf.~\cite[\S 6.2]{BG}, \cite[\S 1.6]{BFGM}) is 
just $(\wbar M(O) \cap M(K))/M(O)$ inside $M(K)/M(O)$, where $O$ is a complete discrete valuation
ring with field of fractions $K$.
The relative version of $\Gr^+_M$ becomes $\eH^+_M$, the positive part of the Hecke stack
(cf.~\cite[\S 1.8]{BFGM}).

The stack $\eH^+_M$ is therefore the global model for the formal arc space of the embedding
$M \into \wbar M$, as considered in \cite[\S 2]{BNY}. We hope that studying the properties
of $\wbar M$ will provide a better understanding of $\eH^+_M$. 

\subsubsection{}
If $P^-$ is a parabolic subgroup opposite to $P$, then $G/U(P)$ is closely related
to the more symmetrically defined variety $X_P = (G/U(P) \xt G/U(P^-))/(P \cap P^-)$, which is also
quasi-affine. This variety $X_P$ is 
called a \emph{boundary degeneration} of $G$ in \cite{SV} (when $P$ is not a Borel subgroup, $X_P$ is an \emph{intermediate} degeneration), and it
is a central object in the geometric proof of Bernstein's Second Adjointness Theorem in 
the theory of $\fp$-adic groups given in \cite{BK}. We note that this proof and 
the space $X_P$ are closely related to the study of geometric constant term (and Eisenstein series) 
functors in \cite{DG:CT}. 

The boundary degeneration $X_P$ and its affine closure $\wbar X_P:=\spec k[X_P]$ may be recovered from 
the \emph{Vinberg semigroup} corresponding to $G$. 
The Vinberg semigroup $\wbar{G_\enh}$ is used to define the Drinfeld-Lafforgue 
compactification $\wbar\Bun_G$ (resp.~the Drinfeld-Lafforgue-Vinberg compactification $\VinBun_G$)
of the moduli stack $\Bun_G$ in \cite{G:Verdier, Schieder:II}. 
As one might expect, 
the positive part $\eH^+_M$ of the Hecke stack appears in the stratification
of $\wbar\Bun_G$ (resp.~$\VinBun_G$), where $P$ ranges over all conjugacy classes of parabolic
subgroups, assuming that $G$ is split.
In this article we attempt to explain the relations between $\wbar M,\, \wbar{G/U(P)},\, 
\wbar X_P$, and $\wbar{G_\enh}$ in hopes that it will elucidate the geometry underlying
the aforementioned stratifications. 

\subsubsection{}
In \cite{Sak}, Sakellaridis fixes a strictly convex cone in the $\bbQ$-vector space spanned
by the coweights of a split maximal torus $T$ in $G$ in order to ``expand power series'' on 
the boundary degeneration $X_P$, under the assumption that the characteristic of $k$ is $0$. 
This cone is precisely the dual of what we call the \emph{Renner cone} of $\wbar M$. 
Thus the combinatorial description of $\wbar M$ provides a first step 
towards generalizing the results of \cite{Sak} to arbitrary characteristic. 

The description of $\wbar M$ is also of interest in the study of those local unramified
automorphic $L$-functions associated to certain ``basic functions'' on $\wbar M$ 
in the spirit of \cite{BNY}. Such functions are considered in \cite{Wang} in relation 
to the asymptotics map\footnote{The asymptotics map, defined in \cite{Sak, SV}, coincides
with the dual of the Bernstein map defined in \cite{BK}.} and inversion of intertwining
operators. The study of $\wbar M$, and more generally of the intermediate boundary degenerations 
$X_P$, is needed in \cite{Wang} to generalize the results of \cite{DW}, which
treats the case when $G = \SL(2)$.

\subsection{Contents}
In \S\ref{sect:recall}, we recall the classification of normal reductive monoids proved by 
L.~Renner. Given a reductive group and certain combinatorial data (what we call a \emph{Renner cone}),
we construct the associated normal algebraic monoid. 

\medskip

In \S\ref{sect:monoid}, we define the normal reductive monoid $\wbar M$ associated
to a parabolic subgroup $P$ of $G$. The group of units of $\wbar M$ is the Levi factor $M$ of
$P$. We first give a combinatorial definition of $\wbar M$ following Renner's classification.
We then show in \S\ref{sect:GUP} that this monoid may be realized as a retract 
of $\wbar{G/U(P)}$, the spectrum of regular functions on the quasi-affine variety $G/U(P)$.
Lastly in \S\ref{sect:Rep(M+)} we describe $\wbar M$ using the Tannakian formalism. This
Tannakian description shows how $\wbar M$ is used implicitly in \cite{BG}, \cite{BFGM}.

\medskip

In \S\ref{sect:XPVin}, we first recall the definition of the boundary degeneration 
$X_P$ associated to a pair of opposite parabolics. We show that 
$\wbar{G/U(P)}$ is a retract (and hence a closed subscheme) of $\wbar X_P:=\spec k[X_P]$. 
Using the relation between the boundary degeneration 
and the Vinberg semigroup of $G$ (i.e., the enveloping semigroup of $G$), we give
another definition of the reductive monoid $\wbar M$ using the 
existence of a certain idempotent in the Vinberg semigroup.

\subsection{Conventions}

Let $k$ be a perfect field of arbitrary characteristic. All schemes considered will be $k$-schemes. 
For a scheme $S$, let $k[S]$ denote the ring of regular functions $\Gamma(S,\eO_S)$.

Fix an algebraic closure $\bar k$ of $k$, and let $\Gal(\bar k/k)$ denote its Galois group.
For a $k$-scheme $S$, let $S_{\bar k}$ denote the base change $S \xt_{\spec k} \spec \bar k$,
and let $\bar k[S] := \Gamma(S_{\bar k}, \eO_{S_{\bar k}})$.

\subsubsection{The group $G$}
Let $G$ be a connected reductive group over $k$.  
Let $T$ denote its \emph{abstract} Cartan and $W$ the corresponding Weyl group. 
We will denote by $\check \Lambda$ (resp.~$\Lambda$) the weight (resp.~coweight) lattice of $T_{\bar k}$,
which is a $\Gal(\bar k/k)$-module.

The semigroup of dominant coweights (resp., weights) will be denoted
by $\Lambda^+_G$ (resp., by $\check \Lambda_G^+$).
The set of vertices of the Dynkin diagram of $G$ will be denoted by $\Gamma_G$; for
each $i \in \Gamma_G$ there corresponds a simple coroot $\alpha_i$
and a simple root $\check \alpha_i$. 
We denote the non-negative integral span of the set of positive coroots (resp.~roots) by $\Lambda^\pos_G$ 
(resp.~$\check\Lambda^\pos_G$). 
For $\lambda,\mu \in \Lambda$ we will write that $\lambda \ge \mu$
if $\lambda - \mu \in \Lambda^\pos_G$, and similarly for $\check\Lambda_G^\pos$.
Let $w_0$ denote the longest element in the Weyl group of $G$.

Let $P$ be a parabolic subgroup of $G$. Let $U(P)$ be its unipotent radical 
and $M:=P/U(P)$ the Levi factor. 
We use $P$ to identify the abstract Cartan of $M$ with $T$ and let $W_M \subset W$ denote the
corresponding Weyl group. 
There is a subdiagram $\Gamma_M \subset \Gamma_G$. 
We will denote by $\Lambda^\pos_M \subset\Lambda_G^\pos$, $\Lambda^+_M \supset \Lambda^+_G$, $\ge_M$, $w_0^M \in W_M$, etc. the corresponding objects for $M$.

\subsubsection{}
Let $\Rep(G)$ denote the abelian category of finite-dimensional $G$-modules. 
This category admits a forgetful functor to the abelian category of $k$-vector spaces.
We define the functor 
\[ \ind^G_P : \Rep(P) \to \Rep(G) \] 
as in \cite[\S I.3.3]{Jantzen}. 
For a $P$-module $\wbar V$, the induced module $\ind^G_P(\wbar V) = (k[G] \ot_k \wbar V)^P$
is finite-dimensional by properness of $G/P$. 
The functor $\ind^G_P$ is right adjoint to the restriction functor 
(cf.~\cite[Proposition I.3.4]{Jantzen}). 
We also denote by $\ind^G_P$ the corresponding functor 
$\Rep(M) \to \Rep(G)$, 
where an $M$-module is considered as a $P$-module with trivial $U(P)$-action.

To a dominant weight $\check\lambda\in \check\Lambda_G^+$ one attaches the Weyl
$G_{\bar k}$-module $\Delta(\check\lambda)$, the dual Weyl module $\nabla(\check\lambda)$,
and the irreducible $G_{\bar k}$-module $L(\check\lambda)$ of highest weight $\check\lambda$.

\subsection{Acknowledgments} 
This research is partially supported by the Department of Defense (DoD) through the National Defense Science and Engineering Graduate Fellowship (NDSEG) Program.
I am very thankful to my doctoral advisor Vladimir Drinfeld for his continual guidance
and support throughout this project. I also thank Simon Schieder for helpful discussions 
on the Vinberg semigroup.

\section{Recollections on normal reductive monoids} \label{sect:recall}
In this section we give a brief review of the classification of normal reductive monoids (i.e., normal, irreducible, affine algebraic monoids whose group of units is reductive),
which is proved in \cite[Theorem 5.4]{Renner} by L.~Renner. 
In \cite{Renner}, the base field is assumed to be algebraically closed, but the statements 
easily generalize to the case of a perfect base field by Galois descent. 

To keep notation consistent with the rest of the article, we consider a connected
reductive group $M$ over $k$. Let $T$ denote its \emph{abstract} Cartan and $W_M$
the corresponding Weyl group.

\subsection{Renner cones}
We denote by $\check\Lambda$ the weight lattice of $T_{\bar k}$ (i.e., the lattice of characters). 
Let $\check\Lambda^\bbQ := \check\Lambda \ot \bbQ$, which is a $\bbQ$-vector space
with a $\Gal(\bar k/k)$-action. 

A \emph{Renner cone} is a convex rational polyhedral cone in $\check\Lambda^\bbQ$ that
is stable under the actions of $W_M$ and $\Gal(\bar k/k)$.
As the name suggests, the theorem of L.~Renner shows that normal algebraic monoids
with group of invertible elements $M$ bijectively correspond to Renner cones generating 
$\check\Lambda^\bbQ$ as a vector space. The correspondence is as follows: 

Let $\wbar M$ be a reductive monoid with group of units $M$.  
Fix a Borel subgroup $B \subset M_{\bar k}$ and a 
Cartan subgroup (i.e., maximal torus) $T_{\sub,\bar k} \subset B$, 
both defined over $\bar k$.
This gives an identification of $T_{\sub,\bar k}$ with the abstract Cartan $T_{\bar k}$. 
Consider the cone $\check C \subset \check\Lambda^\bbQ$ corresponding by 
\cite{KKMS} to the closure of $T_{\sub,\bar k}$ in $\wbar M_{\bar k}$. 
The pairs $(T_{\sub,\bar k}, B)$ of a Cartan subgroup contained in a Borel subgroup are all 
conjugate by $M(\bar k)$.
Since $M$ acts on $\wbar M$ by conjugation, $\check C$ does not depend on the choice
of $(T_{\sub,\bar k}, B)$. 
The Weyl group of $M$ acts on $T_{\sub,\bar k}$ through the normalizer of $T_{\sub,\bar k}$
in $M$, so $\check C$ is preserved by the action of $W_M$ on $\check \Lambda^\bbQ$.
The action of $\Gal(\bar k/k)$ on $G_{\bar k}$ induces an action on the set of pairs 
$(T_{\sub,\bar k}, B)$. Since $\check C$ is canonically defined independently of 
the choice of $(T_{\sub,\bar k}, B)$, the Galois action preserves $\check C$.
Therefore $\check C$ is a Renner cone, and it is the Renner cone corresponding to $\wbar M$.

\subsubsection{} 
Let $\check C \subset \check\Lambda^\bbQ$ be a Renner cone. 
We will construct the corresponding normal reductive monoid $\wbar M$. 
Let us choose a Cartan subgroup (i.e., maximal torus) $T_\sub$ of $M$, defined over $k$.
The construction of $\wbar M$ will not depend on this choice.

\subsection{The monoid $\wbar{T_\sub}$} \label{sect:R}
The characters $\check\Lambda$ form a basis of $\bar k[T]$. Let $R'$
denote the subalgebra of $\bar k[T]$ spanned by the characters in $\check C \cap \check\Lambda$.
The choice of a Borel $B \subset M_{\bar k}$ containing $T_{\sub,\bar k}$ gives
an isomorphism $T_{\sub,\bar k} \cong T_{\bar k}$. All such Borel subgroups are conjugate
by the normalizer of $T_{\sub,\bar k}$ in $M_{\bar k}$. 
The subalgebra $R'$ is preserved by the action of the Weyl group on $T$, so it defines
a corresponding subalgebra $R \subset \bar k[T_\sub]$, which does not depend on the choice of 
a Borel subgroup. 

\subsubsection{} 
Since $\check C$ is Galois stable, so is the subalgebra $R$. 
Set $\wbar{T_\sub} := \spec(R^{\Gal(\bar k/k)})$. 

\begin{lem}\label{lem:toricvar}
	(i) $\wbar{T_\sub}$ is a normal algebraic variety containing $T_\sub$
as a dense open subvariety. 

	(ii) $\wbar{T_\sub}$ has a (unique) monoidal structure extending the
group structure on $T_\sub$. 
\end{lem}
\begin{proof}
By Galois descent, it suffices to check the statements over $\bar k$, 
and we have $\bar k[\wbar{T_\sub}] = R$.
The submonoid $\check C \cap \check \Lambda$ is finitely generated and generates 
$\check \Lambda$ as a group. Moreover the submonoid is \emph{saturated} (i.e., it is
the intersection of a rational cone with the lattice). Statement (i) follows
from \cite[Ch.~1, Thm.~1]{KKMS}. 

To prove statement (ii), one must show that the coproduct map $k[T_\sub] \to k[T_\sub] \ot k[T_\sub]$
sends the subalgebra $R$ to $R \ot R$.
This is clear because $R \ot \bar k$ has a basis consisting of 
characters of $T_{\sub,\bar k}$.
\end{proof}

\subsection{The monoid $\wbar M$}
We will define a normal algebraic monoid $\wbar M$ with group of
units $M$ such that the closure of $T_\sub$ in $\wbar M$ equals $\wbar{T_\sub}$. 
The monoid $\wbar M$ will be the spectrum of a certain
subalgebra $A$ of the algebra of regular functions on $M$.

\subsubsection{The algebra $A$} 
Let $A$ denote the algebra of all $f\in k[M]$ such that
for any $m_1, m_2 \in M(\bar k)$ the function 
\[ t \mapsto f(m_1 t m_2) \] 
belongs to the algebra $R$ defined in \S\ref{sect:R}. Since 
all Cartan subgroups of $M_{\bar k}$ are $M(\bar k)$-conjugate,  $A$ does not
depend on the choice of the subgroup $T_\sub \subset M$.

\begin{prop} \label{prop:A}
	(i) $A$ is a sub-bialgebra of the Hopf algebra $k[M]$. 

	(ii) The map $M \to \spec A$ is an open embedding. 

	(iii) $A$ is an integrally closed domain.

	(iv) The algebra $A$ is finitely generated.

	(v) The homomorphism $A \to k[\wbar{T_\sub}]$ that takes a function to its restriction to $T_\sub$ is surjective.
\end{prop}
\begin{proof} 
All statements can be checked after base change to $\bar k$, so we will assume that $k$ is
algebraically closed.

	Let $A'$ denote the subalgebra of $k[M]$ generated by the matrix coefficients of 
a \emph{finite} collection of 
Weyl\footnote{One can also take dual Weyl modules.} $M$-modules whose
highest weights belong to $\check C \cap \check\Lambda^+_G$ and generate $\check\Lambda^+_G$ 
as a semigroup. The following properties of $A'$ are easy to check:

(a) $A' \subset A$;

(a$'$) if $k$ has characteristic $0$ then $A' = A$;

(a$''$) the morphism $M \to \spec A'$ is an open embedding; 

(b) the composition $A' \into A \to R$ is surjective;

(c) the algebra $A'$ is finitely generated.
\smallskip

\noindent
The proof of statement (i) in the proposition is standard. 
Since $A' \subset A \subset k[M]$, property (a$''$) implies statement (ii). 
Statement (iii) follows from the normality of $M$ and the normality
part of Lemma~\ref{lem:toricvar}(i). Statement (v) follows from (b).

Let $A''$ denote the integral closure of $A'$ in the function field of $M$. 
Without any assumptions on the characteristic of $k$, we claim that
$A = A''$. By (ii)-(iii), it suffices to check that $A$ is contained in the
localization $O$ of $A''$ at any codimension $1$ prime. 
Let $K$ denote the field of fractions of $O$, which is also the field of rational functions on $M$. 
Then the normalization map $\spec A'' \to \spec A'$ induces a map $f' : \spec O \to \spec A'$ 
with $f'(\spec K) \subset M$. 
We wish to lift $f'$ to a morphism $f: \spec O \to \spec A$. Since
\[ M(K) = M(O) \cdot T_\sub(K) \cdot M(O) \]
we can assume that $f'(\spec K) \subset T_\sub(K)$. Then the existence of $f$ follows from (b), 
which says that the closure of $T_\sub$ in $\spec A$ maps isomorphically onto the closure of 
$T_\sub$ in $\spec A'$.
Therefore $A = A''$, and statement (iv) now follows.
\end{proof}


\subsubsection{The algebraic monoid $\wbar M$} Now set $\wbar M := \spec A$. 

By Proposition~\ref{prop:A}, $\wbar M$ is a normal affine algebraic monoid equipped
with an open embedding $M \into \wbar M$ with dense image. 
By part (v) of the proposition, the closed embedding $T_\sub \into M$ extends 
to a closed embedding $\wbar{T_\sub} \into \wbar M$.
By construction, the Renner cone corresponding to $\wbar M$ is $\check C$.

Since $\wbar M$ is an irreducible monoid and $M$ is an open dense subgroup, $M$
is necessarily the group of units of $\wbar M$.  
The classification theorem of L.~Renner (\cite[Theorem 5.4]{Renner}) 
says that every normal algebraic monoid with group of units $M$ is isomorphic to 
a monoid $\wbar M$ of the above form.

\section{The monoid associated to a parabolic subgroup} \label{sect:monoid}
Let $P$ be a parabolic subgroup of $G$ with Levi quotient $M:=P/U(P)$. 
We will define a canonical normal reductive monoid $\wbar M$ with group of units $M$. 
This monoid appears implicitly in \cite{BG, BFGM}, and it is
explicitly considered in \cite[\S 3.3]{AT} (in characteristic $0$) and in \cite[\S 6]{Baranovsky}. 

\smallskip

We identify the abstract Cartans of $G$ and $M$ as follows: 
for a Borel subgroup $B_M \subset M_{\bar k}$, the subgroup $B := B_M U(P) \subset 
G_{\bar k}$ is a Borel subgroup, and $T_{\bar k} = B/U(B) = B_M / U(B_M)$. 

\subsection{The Renner cone of $\wbar M$}
We first give a combinatorial definition of $\wbar M$ using Renner's classification,
recalled in \S\ref{sect:recall}. 
We will specify the Renner cone $\check C \subset \check\Lambda^\bbQ$,
from which one constructs $\wbar M$ as in \S\ref{sect:recall}.

\subsubsection{The submonoid $\Lambda^\pos_{U(P)}$}
Let $\Lambda^\pos_{U(P)} \subset \Lambda$ denote the 
non-negative integral span of the positive coroots of $G$ that are not coroots of $M$.
The submonoid $\Lambda^\pos_{U(P)}$ is stable under the actions of $W_M$ and $\Gal(\bar k/k)$
because $M$ is defined over $k$. 
We use representation theory to show certain properties of $\Lambda^\pos_{U(P)}$ below.

Let $\check G$ (resp.~$\check M$) denote the Langlands dual group of $G$ (resp.~$M$) over $\bbC$.
Fix a maximal torus and a Borel subgroup containing it in the split group $\check G$. 
Then we may consider $\check M$ as a Levi subgroup of $\check G$.
Let $\check\fu_P$ denote the nilpotent Lie algebra 
corresponding to the positive coroots of $G$ that are not coroots of $M$. 
Then the symmetric algebra $\Sym(\check \fu_P)$ is a locally finite
$\check M$-module by the adjoint action,
and its set of weights equals $\Lambda^\pos_{U(P)}$.

\begin{lem} \label{lem:wthull}
	Let $\lambda,\lambda' \in \Lambda^+_M$ with $\lambda \le_M \lambda'$. 
If $\lambda' \in \Lambda^\pos_{U(P)}$, then $\lambda \in \Lambda^\pos_{U(P)}$.
\end{lem}
\begin{proof}
We have a decomposition of $\Sym(\check \fu_P)$ into irreducible highest
weight $\check M$-modules $L_{\check M}(\gamma)$.
Therefore $\lambda'$ is a weight in $L_{\check M}(\gamma)$ for some $\gamma \in \Lambda^+_M$,
and all the weights of $L_{\check M}(\gamma)$ lie in $\Lambda^\pos_{U(P)}$. 
Since $\lambda \in \Lambda^+_M$ and $\lambda \le_M \lambda' \le_M \gamma$, we deduce
that $\lambda$ is also a weight of $L_{\check M}(\gamma)$. 
Therefore $\lambda \in \Lambda^\pos_{U(P)}$.
\end{proof}

\begin{lem} \label{lem:posU}
The subset $\Lambda^\pos_{U(P)}\subset\Lambda$ is equal to the intersection of $w(\Lambda^\pos_G)$
for all $w \in W_M$. Consequently, 
$\Lambda^\pos_{U(P)} \cap (-\Lambda^+_M) = \Lambda^\pos_G \cap (-\Lambda^+_M)$. 
\end{lem}
\begin{proof}
Observe that $\Lambda^\pos_{U(P)}$ is $W_M$-stable and hence contained in 
$w(\Lambda^\pos_G)$ for all $w\in W_M$. 
To prove containment in the other direction, let $\lambda \in 
\bigcap_{w\in W_M} w(\Lambda^\pos_G)$. Replacing $\lambda$ by
an element in the same $W_M$-orbit, we may assume that
$\lambda \in -\Lambda^+_M$. By assumption $\lambda \in \Lambda^\pos_G$,
so we can write $\lambda = \lambda_1 + \lambda_2$ where
$\lambda_1$ is a linear combination of $\alpha_i$ for $i \in \Gamma_M$
 and $\lambda_2 \in \Lambda^\pos_{U(P)}$ is a linear combination
of $\alpha_j$ for $j \in \Gamma_G \setminus \Gamma_M$. 
Note that $\lambda_2 \in \Lambda^\pos_{U(P)} \cap (-\Lambda^+_M)$ and $\lambda \ge_M \lambda_2$. 
Then $w_0^M \lambda_2 \in \Lambda^\pos_{U(P)} \cap \Lambda^+_M$ and 
$w_0^M \lambda \le_M w_0^M \lambda_2$. Lemma~\ref{lem:wthull} implies that 
$w^M_0 \lambda \in \Lambda^\pos_{U(P)}$, and hence 
$\lambda \in \Lambda^\pos_{U(P)}$.
One deduces the second statement of the lemma from the first 
because $\lambda \in -\Lambda^+_M$ satisfies $\lambda \le_M w \lambda$ for all $w\in W_M$.
\end{proof}

\begin{lem} \label{lem:RennerMbar}
The submonoid $W_M \cdot \check\Lambda^+_G \subset \check\Lambda$ is dual to 
$\Lambda^\pos_{U(P)}$, i.e.,
\begin{equation} \label{eqn:sgpdual}
	W_M \cdot \check\Lambda^+_G = 
\{ \check\lambda \in \check\Lambda \mid 
\brac{\check\lambda, \mu} \ge 0 \emph{ for all } \mu \in \Lambda^\pos_{U(P)} \}. 
\end{equation}
\end{lem}
\begin{proof}
	Let $(\Lambda^\pos_{U(P)})^\vee$ equal the r.h.s.~of \eqref{eqn:sgpdual}, which is
evidently $W_M$-stable. If we consider an element in
$(\Lambda^\pos_{U(P)})^\vee\cap \check\Lambda^+_M$, then it pairs with positive coroots of 
$M$ to non-negative integers since the element is $M$-dominant, and it pairs
with all other positive coroots of $G$ to non-negative integers by definition of the dual. 
Thus $(\Lambda^\pos_{U(P)})^\vee\cap \check\Lambda^+_M
= \check\Lambda^+_G$, 
which implies that $(\Lambda^\pos_{U(P)})^\vee$ is the union of $w(\check\Lambda^+_G)$ for
all $w\in W_M$. 
\end{proof}

\begin{cor} \label{cor:saturated}
The submonoid $W_M \cdot \check\Lambda^+_G$ is saturated in $\check\Lambda$.
\end{cor}

\subsubsection{The Renner cone $\check C$}
Set $\check C \subset \check\Lambda^\bbQ$ to be the convex rational polyhedral cone generated by 
$W_M \cdot \check\Lambda^+_G$. Lemma~\ref{lem:RennerMbar} implies that 
$\check C$ is preserved by the actions of $W_M$ and $\Gal(\bar k/k)$, and 
Corollary~\ref{cor:saturated} says that $\check C \cap \check\Lambda = W_M \cdot \check\Lambda^+_G$.

\subsubsection{Definition of $\wbar M$}\label{sect:defineMbar}
Set $\wbar M$ to be the normal reductive monoid with Renner cone $\check C$ 
constructed in Proposition~\ref{prop:A}. We will use this notation for the rest of the
article.

\subsection{Relation to $\wbar{G/U(P)}$} \label{sect:GUP}

In this subsection, we show (see Corollary~\ref{cor:M}) that $\wbar M$ is isomorphic to 
the monoid constructed in \cite[\S 3.3]{AT} and \cite[\S 6]{Baranovsky}. 

First we recall some facts about the homogeneous space $G/U(P)$. 

\subsubsection{} A scheme $S$ is \emph{strongly quasi-affine} if 
the canonical morphism $S \to \spec k[S]$ is an open embedding
and $k[S]$ is a finitely generated $k$-algebra.

F.~D.~Grosshans proved that the quotient variety $G/U(P)$ is strongly quasi-affine in 
\cite{Grosshans}.
Let $\wbar{G/U(P)} = \spec k[G/U(P)]$, where $k[G/U(P)]$ is the subalgebra
of right $U(P)$-invariant regular functions on $G$.

\subsubsection{Weights of $k[G/U(P)]$}   \label{sect:weights}

The Levi factor $M := P/U(P)$ acts on $G/U(P)$ from the right. 
Therefore we can consider $k[G/U(P)]$ as an $M$-module and ask what is the set
of weights\footnote{Let $V$ be an $M$-module over $k$. Choose a Borel subgroup $B_M \subset M_{\bar k}$ and a Cartan subgroup $T_{\sub,\bar k} \subset B_M$, which is isomorphic to 
$T_{\bar k} = B_M/U(B_M)$. We say that the set of weights 
of $V$ is the set of $T_{\sub,\bar k}$-eigenvalues of $V \ot \bar k$. This set 
does not depend on the choice of $(T_{\sub,\bar k}, B_M)$, so it can be considered
as a subset of $\check\Lambda$, which is preserved by $W_M$ and $\Gal(\bar k/k)$.} of this module with respect to the abstract Cartan of $M$.

\begin{lem} \label{lem:weights}
The set of weights of the $M$-module $k[G/U(P)]$ equals 
$W_M \cdot \check\Lambda^+_G \subset \check\Lambda$.
\end{lem}
\begin{proof} We may assume that $k$ is algebraically closed. 
Choose a Borel subgroup $B$ contained in $P$ (so $B/U(P)$ is a Borel subgroup of $M$) and 
let $T_\sub \subset B$ be a maximal torus, which we identify with its image in $M$. 
The weights of $k[G/U(P)]$ are the $T_\sub$-eigenvalues with respect to right translations.
Let $k[G/U(P)]_{\check \gamma},\,\check\gamma \in \check\Lambda$, denote a weight
space.  

Note that $k[G/U(P)]_{\check \gamma}$ is a $G$-module by left translation.
Let $B^-$ denote the opposite Borel subgroup so that $B \cap B^- = T_\sub$.
By unipotence of $U(B^-)$, we deduce that $\check\gamma$ is a weight of $k[G/U(P)]$
if and only if $k[G/U(P)]_{\check \gamma}^{U(B^-)}\ne 0$. 
Hence we are reduced to studying the weight spaces 
of $k[G/U(P)]^{U(B^-)}$. By considering the $T$-action by left translation,
we have decompositions
\[ k[U(B^-)\bs G] = \bigoplus_{\check \lambda \in \check\Lambda^+_G} 
\nabla(\check\lambda),\quad \quad
k[U(B^-)\bs G ]^{U(P)} = \bigoplus_{\check \lambda \in \check\Lambda^+_G} 
\nabla(\check\lambda)^{U(P)} \]
where $U(P)$ acts by right translation. 
Since $U(B^-) P$ is dense in $G$, the restriction from $G$ to $P$ gives an injection 
\[ \nabla(\check\lambda)^{U(P)}
\into \nabla_M(\check\lambda),\] where $\nabla_M(\check\lambda)$ is the
dual Weyl $M$-module. 

We now prove the `only if' direction of the lemma. 
Suppose that $\check \gamma$ is a weight of $k[G/U(P)]$. Then $\check \gamma$ must be a weight of 
$\nabla_M(\check\lambda)$ for some $\check\lambda \in \check\Lambda_G^+$.
There exists $w\in W_M$ such that $w(\check\gamma) \in \check\Lambda_M^+$. Since
the set of weights of $\nabla_M(\check\lambda)$ is $W_M$-stable, 
$w(\check\gamma)$ is also a weight. Hence $w(\check\gamma) \le_M \check\lambda$. 
Since $\brac{\check\alpha_i, \alpha_j} \le 0$
for $i\in \Gamma_M, j\in \Gamma_G \setminus \Gamma_M$, we deduce that 
$w(\check\gamma) \in \check\Lambda^+_G$. 
This proves the `only if' direction of the lemma.

Conversely, suppose $\check\gamma$ is a weight such that $w(\check\gamma) \in \check\Lambda_G^+$
for some $w\in W_M$. Then $\check\lambda:=w(\check\gamma)$ is the highest weight
in $\nabla(\check\lambda)^{U(P)}$. Since the set of weights of an $M$-module
is $W_M$-stable, we conclude that $\check\gamma$ is a weight of $k[G/U(P)]$.
\end{proof}

\begin{cor}\label{cor:Uinv}
	For any $G$-module $V$, the weights of the $M$-module $V^{U(P)}$ are a subset
	of $W_M \cdot \check\Lambda^+_G$. 
\end{cor}
\begin{proof}
Any finite dimensional $G$-module $V$ is a submodule of a direct sum of 
regular representations $k[G]$, so the weights of $V^{U(P)}$
are a subset of the weights of $k[G/U(P)]$. 
\end{proof}

\subsubsection{The closure of $M$ in $\wbar{G/U(P)}$} 
The subgroup $P \subset G$ induces a closed embedding 
\begin{equation} \label{eqn:PsubG} 
	M = P/U(P) \into G/U(P),
\end{equation}
i.e., we embed $M$ in $G/U(P)$ by the right $M$-action on $1\in G$.
Then the closure of $M$ in $\wbar{G/U(P)}$ has the structure of an irreducible algebraic 
monoid\footnote{This monoid is denoted by $M_+$ in \cite{Baranovsky}.}, and the right 
action of $M$ on $G/U(P)$ extends to an action of this monoid
on $\wbar{G/U(P)}$. 
We claim that the normalization of this monoid
is isomorphic to the monoid $\wbar M$ defined in \S\ref{sect:defineMbar}:

\begin{lem} \label{lem:normMbar}
The embedding \eqref{eqn:PsubG} extends to a finite map $\wbar M \to \wbar{G/U(P)}$.
\end{lem}
\begin{proof}
Let $T_\sub$ be a Cartan subgroup of $M$ and embed $T_\sub \into G/U(P)$
using \eqref{eqn:PsubG}. Let $\wbar{T_\sub}$ denote the
closure of $T_\sub$ in $\wbar{G/U(P)}$. By the classification of normal reductive monoids 
in \cite[Theorem 5.4]{Renner}, it suffices to show that the cone corresponding
by \cite{KKMS} to $\wbar{T_\sub}$ is the Renner cone $\check C$ of $\wbar M$.

By definition, $\wbar{T_\sub}$ is the spectrum of the image of the restriction map 
$k[G/U(P)] \to k[T_\sub]$. 
This map is equivariant with respect to right translations by $T_\sub$, so 
$\bar k[\wbar{T_\sub}]$ decomposes into weight spaces.
Let $\check\gamma$ be a weight of $\bar k[G/U(P)]$. 
By left translation by $G$, one can find $f \in \bar k[G/U(P)]_{\check \gamma}$ such that 
$f(1)=1$. Therefore the weights of $k[\wbar{T_\sub}]$ coincide with the weights of $k[G/U(P)]$,
and the claim follows from Lemma~\ref{lem:weights}.
\end{proof}

\subsubsection{} Fix a parabolic subgroup $P^- \subset G$ opposite to $P$.
For the rest of this section we will identify $M$ with the Levi subgroup $P \cap P^-$.

\begin{thm} \label{thm:M}
The composition 
\[ \wbar M \to \wbar{G/U(P)} \to \spec k[G]^{U(P^-) \xt U(P)} \] 
is an isomorphism, where  
$U(P^-) \xt U(P)$ acts on $k[G]$ by left and right translations, respectively. 
\end{thm}

Note that $\spec k[G]^{U(P^-) \xt U(P)}$ is the affine GIT quotient of $\wbar{G/U(P)}$ by 
the left action of $U(P^-) \subset G$. 

\begin{cor} \label{cor:M}
	The (unique) map $\wbar M \to \wbar{G/U(P)}$ extending the embedding \eqref{eqn:PsubG} 
is a retract. In particular, it is a closed embedding.
\end{cor}
\begin{proof} The fact that $\wbar M$ is a retract of $\wbar{G/U(P)}$ follows
	immediately from the isomorphism in Theorem~\ref{thm:M}. 
To prove that it is a closed subscheme, it suffices to show that the algebra map $k[G/U(P)] \to k[\wbar M]$ is surjective.
The theorem implies that the subalgebra $k[G]^{U(P^-) \xt U(P)} \subset k[G/U(P)]$
surjects onto $k[\wbar M]$.
\end{proof}

For the purpose of proving Theorem~\ref{thm:M}, let $\tilde M = \spec k[G]^{U(P^-) \xt U(P)}$. 
The actions of $M$ on $G$ by left and right translations induce $M$-actions on $\tilde M$.
We have a canonical $M \xt M$ equivariant map $G \to \tilde M$.

\begin{lem} \label{lem:Mopen}
The composition $M \to G \to \tilde M$ is an open embedding.
\end{lem}
\begin{proof} We may check the assertion after base change to $\bar k$, so 
we assume $k$ is algebraically closed. Choose Borel subgroups $B \subset P$ and 
$B^- \subset P^-$ such that $T_\sub := B \cap B^- \subset M$ is a maximal torus. 
Let \[ \tilde T = \spec k[G]^{U(B^-) \xt U(B)}.\]
Since $k[G]^{U(B^-)}$ has a decomposition into dual Weyl $G$-modules, one 
deduces that $k[\tilde T]$ has a basis formed by $f_{\check \lambda}$ for 
$\check\lambda \in \check\Lambda^+_G$, where $f_{\check\lambda}(t)=\check\lambda(t),\,
t\in T_\sub$. 
From this explicit description, one sees that $\tilde T$ is a toric variety containing 
$T_\sub$ as a dense open subscheme.

Consider the composition $G \to \tilde M \to \tilde T$ and 
let $\overset\circ G \subset G$ denote the preimage of $T_\sub \subset \tilde T$. 
Then the preimage of $T_\sub$ in $\tilde M$, which we denote $\overset\circ{\tilde M}$, 
is equal to $\spec k[\overset\circ G]^{U(P^-) \xt U(P)}$.
Observe that $B^-B = U(B^-) \xt T_\sub \xt U(B)$ is an open affine subset contained
in $\overset\circ G$. Let us show that $\overset \circ G = B^-B$. 
By definition, $\overset\circ G$ consists of $g \in G$ such that 
$f_{\check\lambda}(g) \ne 0$ for all dominant weights $\check\lambda$. 
By the Bruhat decomposition, it suffices to show that if $w$ belongs to the
normalizer of $T_\sub$ but not to $T_\sub$ (i.e., $w$ corresponds to a nontrivial element of $W$), 
then there exists $\check\lambda$ with $f_{\check\lambda}(w)=0$. 
Indeed, for a dominant regular weight $\check\lambda$ we have $w\check\lambda \ne \check\lambda$.
Thus the left and right $T$-actions on $w^{-1}f_{\check\lambda}$ do not have the same
weight, which implies that $f_{\check\lambda}(w)=(w^{-1} f_{\check\lambda})(1) = 0$.

Let $B_M = B/U(P) = B \cap M$ and $B_M^- = B^-/U(P^-) = B^- \cap M$. 
From the equality $\overset\circ G = U(B^-)\xt T_\sub \xt U(B)$ we deduce that
$\overset\circ{\tilde M} = U(B^-_M) \xt T_\sub \xt U(B_M)$ is an open dense subset of both $M$ and $\tilde M$.  
Using left (or right) translations by $M$, we deduce that the whole group $M$ is an open 
subset of $\tilde M$.
\end{proof}

The fraction field of $\tilde M$ is contained in\footnote{In fact one can show that the fraction field of $k[G]^{U(P^-) \xt U(P)}$ is equal to $k(G)^{U(P)^- \xt U(P)}$.} the field of invariants $k(G)^{U(P^-) \xt U(P)}$.
Thus normality of $G$ implies normality of $\tilde M$.
Therefore Lemma~\ref{lem:Mopen} implies that $\tilde M$ is a normal reductive monoid with group of 
units $M$.

\begin{proof}[Proof of Theorem~\ref{thm:M}] 
Let $T_\sub$ be a Cartan subgroup of $M \subset G$. 
Since $\tilde M$ is a normal reductive monoid with group of units $M$, it is determined by 
the closure of $T_\sub$ in $\tilde M$, which is the spectrum of
the algebra 
\[ \tilde R := \im(k[G]^{U(P^-) \xt U(P)} \to k[T_\sub]).\]
By unipotence of $U(P^-)$, the algebra $\tilde R$ is the image of 
the restriction map $k[G/U(P)] \to k[T_\sub]$. 
Therefore $\spec \tilde R$ is the closure of $T_\sub$ in $\wbar{G/U(P)}$. 
By the proof of Lemma~\ref{lem:normMbar}, this is also the closure of $T_\sub$ in $\wbar M$.
Since $\tilde M$ and $\wbar M$ are both normal algebraic monoids with unit group $M$,
the classification of normal reductive monoids (\cite[Theorem 5.4]{Renner})
implies that the map $\wbar M \to \tilde M$ is an isomorphism.
\end{proof}

\subsection{Tannakian description of $\wbar M$} \label{sect:Rep(M+)}

Let $\Rep(M)$ denote the monoidal category of finite-dimensional representations of $M$.
Similarly, one has the monoidal category $\Rep(\wbar M)$. Since $M$ is schematically
dense in $\wbar M$, the monoidal functor
\[ \Rep(\wbar M) \to \Rep(M) \]
corresponding to $M \into \wbar M$ is fully faithful. So we can consider 
$\Rep(\wbar M)$ as a full subcategory of $\Rep(M)$.

\subsubsection{} The usual Tannakian formalism describes $\wbar M$ in terms of $\Rep(\wbar M)$.
Namely, for a test scheme $S$, an element of the monoid $\Hom(S,\wbar M)$ is a collection of 
assignments 
\[ \wbar V \in \Rep(\wbar M) \rightsquigarrow m_{\wbar V} \in \End_{\eO_S}(\wbar V \ot \eO_S), \]
compatible with morphisms $\wbar V_1 \to \wbar V_2$ in $\Rep(\wbar M)$ and such that
$m_{\wbar V_1 \ot \wbar V_2} = m_{\wbar V_1} \ot m_{\wbar V_2}$. 
The multiplication in $\Hom(S,\wbar M)$ corresponds to the multiplication in $\End_{\eO_S}(\wbar V \ot \eO_S)$.

Our goal is to prove Proposition~\ref{prop:repM+} below, which describes the subcategory
$\Rep(\wbar M)$.

\subsubsection{Description of $\Rep(\wbar M)$}
Fix a parabolic subgroup $P^-\subset G$ opposite to $P$, 
and identify the Levi subgroup $P \cap P^-$ with $M$. 

For an $M$-module $\wbar V$, 
we consider an element $f \in \ind^G_{P^-}(\wbar V)$ as a regular map 
$G \to \wbar V$ (cf.~\cite[\S I.3.3]{Jantzen}) satisfying $f(g m \bar u) = m^{-1} f(g)$ 
for all $\bar k$-points $g\in G, m \in M, \bar u \in U(P^-)$.
Using this description, evaluation at $1$ in $G$ defines 
an $M$-morphism $\ind^G_{P^-}(\wbar V) \to \wbar V$. 

\begin{lem}\label{lem:Mind}
Let $\wbar V \in \Rep(\wbar M)$. Then evaluation at $1$ induces an isomorphism
\begin{equation} \label{eqn:indres}
	\ind^G_{P^-}(\wbar V)^{U(P)} \to \wbar V.
\end{equation}
\end{lem}

\begin{proof}
Since $U(P) P^-$ is a dense open subset of $G$, the map
\eqref{eqn:indres} is injective. 
Let $v \in \wbar V$. Then we can define a morphism 
$f : U(P) \xt M \xt U(P^-) \cong U(P) P^- \to \wbar V$ by
$f(u m \bar u) = m^{-1} v$ for $m \in M,\, u\in U(P),\, \bar u \in U(P^-)$.
For any $\xi \in \wbar V^*$, the pairing $\brac{\xi, f(u m \bar u)} = \brac{\xi, m^{-1} v}$ 
extends to a regular function in $k[G]^{U(P) \xt U(P^-)}$
by Theorem~\ref{thm:M}.
Therefore $f$ extends to a $U(P)$-invariant function in $\ind^G_{P^-}(\wbar V)$,
proving surjectivity of \eqref{eqn:indres}. 
\end{proof}

\begin{prop}\label{prop:repM+} Let $\wbar V \in \Rep(M)$. Then the following are equivalent:

(i) $\wbar V$ belongs to $\Rep(\wbar M)$. 

(ii) The weights of $\wbar V$ lie in $W_M \cdot \check\Lambda^+_G \subset \check\Lambda$.

(iii) There exists $V \in \Rep(G)$ such that $\wbar V \cong V^{U(P)}$. 
\end{prop}

\begin{proof}
The equivalence of (i) and (ii) follows from the definition of 
$\wbar M$ in \S\ref{sect:defineMbar}.
Corollary~\ref{cor:Uinv} proves (iii) implies (ii). 
Lemma~\ref{lem:Mind} shows that (i) implies (iii) by setting $V = \ind^G_{P^-}(\wbar V)$,
which is a finite-dimensional $G$-module.
\end{proof}

\begin{rem} Suppose that $k$ is algebraically closed.
One deduces from Lemma~\ref{lem:Mind} that $\nabla(\check\lambda)^{U(P)}$ is
isomorphic to the dual Weyl $M$-module $\nabla_M(\check\lambda)$. 
By \cite[Remark II.2.11]{Jantzen}, the subspace $\nabla(\check\lambda)^{U(P)}$
also equals the sum of the weight spaces of $\nabla(\check\lambda)$ with weights 
$\le_M \check\lambda$. 
Dually, one sees that the sum of the weight spaces of $\Delta(\check\lambda)$
with weights $\le_M \check\lambda$ is isomorphic to $\Delta(\check\lambda)_{U(P^-)}$, 
which is in turn isomorphic to the Weyl $M$-module $\Delta_M(\check\lambda)$.
\end{rem}

\begin{rem}
Let $O$ be a complete discrete valuation ring with field of fractions $K$ and residue field $k$.
By Proposition~\ref{prop:repM+}(iii) and the usual Tannakian formalism, 
one observes that the closed subscheme $\Gr^+_M \subset \Gr_M = M(K)/M(O)$  
defined in \cite[\S 6.2]{BG}, \cite[\S 1.6]{BFGM} is 
equal to the subspace $(\wbar M(O) \cap M(K))/M(O)$.
\end{rem}

\section{Relation to boundary degenerations} \label{sect:XPVin}
Let $P$ and $P^-$ be a pair of opposite parabolic subgroups in $G$. We identify 
the Levi subgroup $P \cap P^-$ with the Levi factor $M= P/U(P)$. 
Let $\wbar M$ be the normal reductive monoid with group of units $M$ 
defined in \S\ref{sect:defineMbar}.

In this section we will show that $\wbar{G/U(P)}$ embeds as a closed subscheme
in the affine closure of the boundary degeneration defined in \cite{BK, SV, Sak}. 
We will also describe the relation between the boundary degeneration and 
the Vinberg semigroup (i.e., enveloping semigroup) of $G$. 
This will give an alternate description of $\wbar M$
as a subscheme of the Vinberg semigroup, using idempotents.

\subsection{Boundary degenerations} \label{sect:X_P}
Define the boundary degeneration 
\[ X_P := (G \xt G)/(P \xt_M P^-) = (G/U(P) \xt G/U(P^-))/(P \cap P^-), \] 
where $P \cap P^-$ acts diagonally on the right.
This is a strongly quasi-affine variety by \cite{Grosshans} and Hilbert's theorem
on reductive groups. 

\begin{rem}\label{rem:X_Pspherical}
The group $G \xt G$ acts on $X_P$ by left translations. Suppose that $k$
is algebraically closed and choose a pair $B,B^-$ of  
opposite Borel subgroups contained in $P,P^-$, respectively.
Then the orbit of $B^- \xt B$ acting on $(1,1) \in X_P$ is a dense open subset. 
Therefore $X_P$ is a spherical variety with respect to $G \xt G$.
\end{rem}

\subsubsection{} Let $\wbar X_P = \spec k[X_P]$.
Since $X_P$ is strongly quasi-affine, $\wbar X_P$ is affine of finite type
and the canonical embedding $X_P \into \wbar X_P$ is open.

Note that $\wbar X_P$ is the affine GIT quotient of $\wbar{G/U(P)} \xt \wbar{G/U(P^-)}$
by the diagonal right $M$-action, but it is \emph{not} the stack quotient.

\subsubsection{}
Consider the map of strongly quasi-affine varieties
\begin{equation}\label{eqn:GsubX} 
	G/U(P) \to X_P : g \mapsto (g, 1). 
\end{equation}
The base change of \eqref{eqn:GsubX} under the smooth cover $G \xt G \to X_P$ 
gives the natural closed embedding $G \xt P^- \into G \xt G$. 
Therefore \eqref{eqn:GsubX} is also a closed embedding.

The composition $G/U(P) \into X_P \into \wbar X_P$ induces a map 
\begin{equation}\label{eqn:GsubXbar}
	\wbar{G/U(P)} \to \wbar X_P.
\end{equation}
In characteristic $0$, one easily deduces from \cite[Proposition 5]{AT} 
that \eqref{eqn:GsubXbar} is a closed embedding. In positive characteristic, 
this is not \emph{a priori} clear, but the following theorem shows it is still true:

\begin{thm} \label{thm:X_P}
	The map \eqref{eqn:GsubXbar} is a closed embedding, and the composition 
\[ \wbar{G/U(P)} \to \wbar X_P \to \spec k[X_P]^{U(P)} \]
is an isomorphism, where $U(P) \subset G$ acts on $X_P$ by left translations in the second
coordinate.
\end{thm}

\begin{proof}
Observe that $k[X_P]^{U(P)} = (k[G/U(P)] \ot k[G]^{U(P)\xt U(P^-)})^M$ where $M$ acts diagonally
by right translations. Using the inversion operator on $G$ in the second coordinate, 
we get $k[X_P]^{U(P)} \cong (k[G/U(P)] \ot k[\wbar M])^M$ where $k[\wbar M] = k[G]^{U(P^-) \xt U(P)}$
by Theorem~\ref{thm:M} and $M$ acts anti-diagonally on the right. 
Since $M$ is dense in $\wbar M$, the evaluation at $1 \in \wbar M$ gives
an injection $(k[G/U(P)] \ot k[\wbar M])^M \into k[G/U(P)]$.
On the other hand, $\wbar M$ is the closure of $M$ in $\wbar{G/U(P)}$ by 
Corollary~\ref{cor:M}. The right action of $M$ on $G/U(P)$ therefore extends to a right action 
of $\wbar M$ on $\wbar{G/U(P)}$, which corresponds to a comodule map 
$k[G/U(P)] \to (k[G/U(P)] \ot k[\wbar M])^M$.
The composition 
\[ k[G/U(P)] \to (k[G/U(P)] \ot k[\wbar M])^M \into k[G/U(P)] \]
is the identity, which proves that the composition $\wbar{G/U(P)} \to \spec k[X_P]^{U(P)}$
is an isomorphism. 
It follows that the affine map \eqref{eqn:GsubXbar} is a closed embedding.
\end{proof}

\begin{cor}\label{cor:MX_P}
Consider the embedding $M \into X_P$ defined as the composition of the embeddings 
\eqref{eqn:PsubG} and \eqref{eqn:GsubX}. The
closure of $M$ in $\wbar X_P$ is isomorphic to $\wbar M$. 
The composition 
\[ \wbar M \to \wbar X_P \to \spec k[X_P]^{U(P^-) \xt U(P)} \]
is an isomorphism, where $U(P^-) \xt U(P) \subset G \xt G$ acts on $X_P$ by left translations.
\end{cor}
\begin{proof}
	Combine Theorems~\ref{thm:M} and \ref{thm:X_P}.
\end{proof}

\subsection{Relation to Vinberg's semigroup} Recall that $k$ is an arbitrary perfect field.

\subsubsection{}
We first give a brief review of the standard material on the Vinberg semigroup, which 
is contained in 
\cite{Vinberg,Ritt,Renner}. 

Let $Z(G)$ denote the center of $G$. Consider the group 
\[ G_\enh := (G \xt T)/Z(G),\] 
where $Z(G)$ maps to $G \xt T$ anti-diagonally. Note that $Z(G_\enh) = T$.

\smallskip

The Vinberg semigroup of $G$, denoted $\wbar{G_\enh}$, is a normal reductive
$k$-monoid with group of units $G_\enh$. 
The Renner cone of $\wbar{G_\enh}$ is by definition
\begin{equation} \label{eqn:RenVin}
	\{ (\check\lambda_1, \check\lambda_2) \in \check\Lambda^\bbQ \xt \check\Lambda^\bbQ 
	\mid \check\lambda_2 - w \check\lambda_1 \in \check\Lambda^{\pos,\bbQ}_G \text{ for all } w \in W \},
\end{equation}
where $\check\Lambda^{\pos,\bbQ}_G$ is the rational polyhedral cone generated by
the positive roots of $G$. 
The Vinberg semigroup may be constructed from the Renner cone 
as described in \S\ref{sect:recall}.

The canonical homomorphism of algebraic groups $G_\enh \to T_\adj:=T/Z(G)$ extends to 
a homomorphism of algebraic monoids
\[ \bar \pi : \wbar{G_\enh} \to \wbar{T_\adj}, \] 
where $\wbar{T_\adj} := \mf{t}_\adj$ is the Cartan Lie algebra of the adjoint group.
Let $\overset\circ{\wbar{G_\enh}}$ denote the non-degenerate locus of $\wbar{G_\enh}$.
It is known that $\overset\circ{\wbar{G_\enh}}$ is smooth over $\wbar{T_\adj}$.

\subsubsection{}
For a parabolic $P$ with Levi factor $M$, let $\mathbf c_P \in \wbar{T_\adj}$ be the point
defined by the condition that $\check\alpha_i(\mathbf c_P) = 1$ for simple roots $\check\alpha_i,\, i\in \Gamma_M$, and $\check\alpha_j(\mathbf c_P)=0$ for all other simple roots. 
Note that $\mathbf c_P$ is an idempotent with respect to the monoid structure
on $\wbar{T_\adj}$.

Let $\wbar{G_\enh}_{,\mathbf c_P}$ denote the fiber of $\bar \pi$ over $\mathbf c_P$. 
Note that by definition of $\mathbf c_P$, the center $Z(M)$ is the stabilizer of $T$ acting 
on $\mathbf c_P$ in $\wbar{T_\adj}$.

\subsubsection{}

Fix a pair of opposite parabolic subgroups $P$ and $P^-$, and identify $M$ with 
the Levi subgroup $P \cap P^-$. Since conjugation by $M$ fixes $Z(M)$, the center of $M$
can be embedded as a subgroup of the abstract Cartan $T$. 
Consider the anti-diagonal map  
\[ \fs:	Z(M)/Z(G) \to (Z(M) \xt T)/Z(G) \into (G \xt T)/Z(G)= G_\enh 
\]
defined by $\fs(t) = (t^{-1},t)$.
Observe that $Z(M)/Z(G) \subset T/Z(G) = T_\adj$ coincides with the subtorus 
$\{ t \in T_\adj \mid \check\alpha_i(t) = 1,\, i \in \Gamma_M \}$. 
Let $\wbar{Z(M)/Z(G)}$ denote the closure of $Z(M)/Z(G)$ in $\wbar{T_\adj}$. 

\begin{lem}
(i) The map $\fs$ extends to a homomorphism 
\[ \bar\fs : \wbar{Z(M)/Z(G)}\to \wbar{G_\enh} \]
of algebraic monoids. 

(ii) The composition $\bar\pi \circ \bar\fs$ is the natural inclusion $\wbar{Z(M)/Z(G)} \into \wbar{T_\adj}$.
\end{lem}
\begin{proof} Since we know the composition $\pi \circ \fs$, it suffices 
	to prove statement (i). 
We may assume that $k$ is algebraically closed. 

The weight lattice of $Z(M)/Z(G)$ is the free abelian group $\check\Lambda_{Z(M)/Z(G)}$ 
with basis consisting of the simple roots $\check\alpha_j$ for $j \in \Gamma_G \setminus \Gamma_M$. 
If $\check\lambda = \sum_{i\in \Gamma_G} n_i \check \alpha_i$ for $n_i \in \bbZ$, 
let $\pr(\check\lambda) := \sum_{j\notin \Gamma_M} n_j \check\alpha_j$. 
Let $\check C$ denote the Renner cone \eqref{eqn:RenVin} of $\wbar{G_\enh}$,
and let $\check C_\bbZ := \check C \cap (\check \Lambda \xt \check\Lambda)$. 
Fix a Cartan subgroup $T_\sub \subset M$ and identify $T_\sub$ with $T$ by 
choosing a Borel. The map $\fs$ lands in $(T_\sub\xt T)/Z(G)$, so 
we have an induced map of weights (restricted to $\check C_\bbZ$): 
\[ \check C_\bbZ \to \check\Lambda_{Z(M)/Z(G)} : (\check\lambda_1,\check\lambda_2)\mapsto 
\pr(\check\lambda_2 - \check\lambda_1).\] 
The image of this map is the non-negative span of the simple roots $\check\alpha_j,
j\notin\Gamma_M$. Statement (i) follows. 
\end{proof}

\begin{rem}
If $k$ is algebraically closed, then the map $\bar\fs$ we have constructed
factors through the section $\wbar{T_\adj} \to \oo{\wbar{G_\enh}}$ constructed
in \cite[Lemma D.5.2]{DG:CT}, which depends on a choice of Borel subgroup and maximal torus
of $G$. 
In particular, $\bar\fs$ always lands in the non-degenerate locus of the Vinberg
semigroup for arbitrary $k$.

\end{rem}

\subsubsection{The idempotent $e_P$}
Observe that $\mathbf c_P$ lies in the submonoid $\wbar{Z(M)/Z(G)} \subset \wbar{T_\adj}$. 
Define the idempotent 
\[ e_P := \bar\fs(\mathbf c_P) \in \oo{\wbar{G_\enh}}(k),\] 
which satisfies $\bar\pi(e_P) =\mathbf c_P$. 
In \cite[Appendix C]{DG:CT},
it is shown (by passing to an algebraic closure $\bar k$) that 
\[ P = \{ g \in G \mid g \cdot e_P = e_P\cdot g\cdot e_P \} \quad\text{and}\quad 
	P^- = \{ g \in G \mid e_P\cdot g = e_P\cdot g\cdot e_P \}, \]
and the stabilizer of the $P \xt P^-$ action on $e_P$ equals $P \xt_M P^-$. 
Note that if $g \in P \cap P^-$, then $g \cdot e_P = e_P \cdot g \cdot e_P = e_P \cdot g$.
It follows that $M$ is the centralizer of $e_P$ in $G$. 

\begin{rem} \label{rem:GeG}
	It is known that $G \cdot e_P \cdot G$ is equal to the
non-degenerate locus $\oo{\wbar{G_\enh}}_{,\mathbf c_P}$ of the fiber. 
One deduces from the above that the $G \xt G$-action on $e_P$ induces
an isomorphism\footnote{In fact, we learned from S.~Schieder that this induces
an isomorphism of affine varieties $\wbar X_P \cong \wbar{G_\enh}_{,\mathbf c_P}$.}
\[ X_P := (G \xt G)/(P \xt_M P^-) \cong \oo{\wbar{G_\enh}}_{,\mathbf c_P}.\]
\end{rem}

\begin{rem} 
Suppose that $k$ is algebraically closed.
By a result of M.~Putcha (cf.~\cite[Theorem 4.5]{Renner}) for general reductive monoids, any 
idempotent in the non-degenerate locus of $\wbar{G_\enh}$ is $G(k)$-conjugate to
$e_P$ for some parabolic $P$. Moreover, the choice of $P$ and $P^-$ determines this
idempotent in its conjugacy class. 
\end{rem}

\subsubsection{Relating $\wbar M$ to the Vinberg semigroup}
Consider the map 
\begin{equation}  \label{eqn:ege}
	G \to \wbar{G_\enh}_{,\mathbf c_P} : g \mapsto e_P \cdot g \cdot e_P.
\end{equation}
Since $U(P) \cdot e_P = e_P \cdot U(P^-) = \{e_P\}$, this map is $U(P^-) \xt U(P)$-invariant. 
By Theorem~\ref{thm:M}, we have an isomorphism $\wbar M \cong \spec k[G]^{U(P^-) \xt U(P)}$. 
Since $\wbar{G_\enh}_{,\mathbf c_P}$ is affine, the map \eqref{eqn:ege}
must factor through a map
\begin{equation} \label{eqn:Vin}
	\wbar M \to e_P \cdot \wbar{G_\enh}_{,\mathbf c_P} \cdot e_P. 
\end{equation}
Observe that $e_P \cdot \wbar{G_\enh}_{,\mathbf c_P} \cdot e_P$ is an irreducible
algebraic monoid with identity $e_P$. 
The map \eqref{eqn:Vin} is an extension of the homomorphism 
of algebraic monoids $M \to \wbar{G_\enh}_{,\mathbf c_P}$ sending 
$m\mapsto m \cdot e_P = e_P \cdot m$. 
Therefore \eqref{eqn:Vin} must also be a homomorphism of algebraic monoids.

\begin{thm}\label{thm:Vin}
	The homomorphism \eqref{eqn:Vin} is an isomorphism of algebraic monoids. 
\end{thm}

By the definition of \eqref{eqn:ege}, we see that the image of \eqref{eqn:Vin}
contains $e_P \cdot G \cdot e_P$. Since the latter map is a homomorphism of monoids, 
we deduce that the image contains $e_P \cdot G \cdot e_P \cdot G \cdot e_P$. 
By Remark~\ref{rem:GeG}, we have $G \cdot e_P \cdot G = \oo{\wbar{G_\enh}}_{,\mathbf c_P}$
is dense in $\wbar{G_\enh}_{,\mathbf c_P}$. Multiplying on the left and right
by $e_P$, we deduce that \eqref{eqn:Vin} has dense image.
On the other hand, the restriction of \eqref{eqn:Vin} to $M$ is injective. 
It follows that $M\cdot e_P$ is a dense sub\emph{group} of $e_P \cdot \wbar{G_\enh}_{,\mathbf c_P}
\cdot e_P$. Therefore $M\cdot e_P$ must be equal to the group of units of 
$e_P \cdot \wbar{G_\enh}_{,\mathbf c_P} \cdot e_P$. 

\medskip

We show that the monoid $e_P \cdot \wbar{G_\enh}_{,\mathbf c_P} \cdot e_P$ is normal and 
then use Renner's classification of normal monoids to prove the theorem. 

Consider the larger algebraic monoid
$e_P \cdot \wbar{G_\enh} \cdot e_P$ with unit $e_P$ (where we do not restrict to a fiber). 
The action of $Z(G_\enh) = T$ on $e_P \cdot \wbar{G_\enh} \cdot e_P$ induces an 
isomorphism
\begin{equation}\label{eqn:fibration}
	((e_P \cdot \wbar{G_\enh}_{,\mathbf c_P} \cdot e_P) \xt T)/Z(M) \cong 
e_P \cdot \wbar{G_\enh} \cdot e_P,
\end{equation}
so the two aforementioned monoids are closely related.

Since $e_P \cdot \wbar{G_\enh} \cdot e_P$ is the closed subscheme of 
the Vinberg semigroup fixed by left and right multiplications by $e_P$, it is a retract
of $\wbar{G_\enh}$ in the category of schemes. The retraction is 
given by the formula $x \mapsto e_P \cdot x \cdot e_P$.

\begin{lem}
Let $Y$ and $S$ be integral affine schemes such that $Y$ is a retract of $S$ (i.e., 
there exist maps $Y \to S$ and $S \to Y$ such that their composition is the
identity map on $Y$). If $S$ is normal then so is $Y$.
\end{lem}
\begin{proof}
Since $Y$ is a retract of $S$, we have an inclusion of algebras $k[Y] \into k[S]$. 
The algebra $k[S]$ is integrally closed, so if $\tilde Y$ denotes the
normalization of $Y$ in its field of fractions, then the previous inclusion factors 
as $k[Y] \into k[\tilde Y] \into k[S]$. 
On the other hand the map $Y \to S$ induces an algebra map $k[S] \to k[Y]$
which restricts to the identity on $k[Y]$. Localization implies that
the composition $k[\tilde Y] \to k[Y]$ is injective and hence an isomorphism.
\end{proof}

\begin{cor} \label{cor:ePnorm}
The algebraic monoid $e_P \cdot \wbar{G_\enh} \cdot e_P$ is normal.
\end{cor}
\begin{proof}
	The Vinberg semigroup is normal by definition, and 
we have observed that $e_P \cdot \wbar{G_\enh} \cdot e_P$ is a retract of $\wbar{G_\enh}$.
\end{proof}

\begin{cor} \label{cor:ePnormf}
	The algebraic monoid $e_P \cdot \wbar{G_\enh}_{,\mathbf c_P} \cdot e_P$ is normal.
\end{cor}
\begin{proof}
	We deduce from \eqref{eqn:fibration} that 
$e_P \cdot \wbar{G_\enh} \cdot e_P$ is smooth locally isomorphic to 
$(e_P \cdot \wbar{G_\enh}_{,\mathbf c_P} \cdot e_P) \xt T$. 
It follows from Corollary~\ref{cor:ePnorm} and ascending and descending properties of
normality that $e_P \cdot \wbar{G_\enh}_{,\mathbf c_P} \cdot e_P$ is normal.
\end{proof}

\begin{proof}[Proof of Theorem~\ref{thm:Vin}] 
	By Corollary~\ref{cor:ePnormf} we know that $e_P \cdot \wbar{G_\enh}_{,\mathbf c_P} \cdot
e_P$ is a normal reductive monoid with group of units $M\cdot e_P$.
Recall from \S\ref{sect:recall} that normal reductive monoids are classified by their Renner
cones. 
Since $\wbar M$ is also a normal reductive monoid with group of units $M$, 
to prove the theorem it suffices 
to check that the Renner cones of $\wbar M$ and $e_P \cdot \wbar{G_\enh}_{,\mathbf c_P} \cdot e_P$
are equal.  We may assume that $k$ is algebraically closed.

Fix a Cartan subgroup $T_\sub \subset M \subset G$. Identify $T_\sub$ with the abstract
Cartan $T$ by choosing a Borel subgroup. 
Consider the embedding $T_\sub \into \wbar{G_\enh}$ sending $t \mapsto t \cdot e_P$
and let $\wbar{T_\sub \cdot e_P}$ denote the closure of the image. 
Set $T_\enh := (T_\sub \xt T)/Z(G)$, which is a Cartan subgroup of $G_\enh$, and
let $\wbar{T_\enh}$ denote its closure in $\wbar{G_\enh}$. 
By definition, $e_P$ lies in $\wbar{T_\enh}$, so $T_\sub \into \wbar{G_\enh}$
factors through the homomorphism of monoids
\begin{equation} \label{eqn:Tsubenh}
	T_\sub \into \wbar{T_\enh}  : t \mapsto t \cdot e_P
\end{equation}
Let $\check C \subset \check\Lambda^\bbQ \xt \check\Lambda^\bbQ$ denote the
Renner cone \eqref{eqn:RenVin} of $\wbar{G_\enh}$. 
Recall that the weights in $\check C_\bbZ:=\check C \cap (\check\Lambda \xt \check\Lambda)$ form 
a basis of $k[\wbar{T_\enh}]$.  
Let $(\check\lambda_1,\check\lambda_2) \in \check C_\bbZ$. 
Then $\check\lambda_2 - \check\lambda_1 \in \check\Lambda^\pos_G$, so it may be
considered as a regular function on $\wbar{T_\adj}$. 
Evaluating this function at $\mathbf c_P$ gives a number $(\check\lambda_2 - \check\lambda_1)(\mathbf c_P)$, which is $1$ if $\check\lambda_2 -\check\lambda_1 \in \check\Lambda^\pos_M$ and $0$ otherwise.
By the definition of $e_P$, one sees that the 
homomorphism \eqref{eqn:Tsubenh} corresponds to the map of weights 
\begin{equation}\label{eqn:lattice}
	\check C_\bbZ \to \check\Lambda  : (\check\lambda_1,\check\lambda_2) \mapsto 
(\check\lambda_2 - \check\lambda_1)(\mathbf c_P) \cdot \check\lambda_1. 
\end{equation}
The existence of the map \eqref{eqn:Vin} implies that 
the image of \eqref{eqn:lattice} must land in the Renner cone of $\wbar M$, which 
is generated by the saturated submonoid $W_M \cdot \check\Lambda^+_G$. 
On the other hand, for $\check\lambda \in \check\Lambda^+_G$ and $w \in W_M$, 
one sees that $(w\check\lambda, \check\lambda) \mapsto w\check\lambda$. 
Thus the image of \eqref{eqn:lattice} equals $W_M \cdot \check\Lambda^+_G$.

Therefore the Renner cones of $\wbar M$ and $e_P \cdot \wbar{G_\enh}_{,\mathbf c_P} \cdot e_P$
are equal, which proves the theorem.
\end{proof}

\bibliographystyle{amsplain}
\bibliography{monoid}

\end{document}